\documentclass[11pt]{amsart}%
\usepackage{hyperref}
\usepackage{amsmath}%
\usepackage{amsfonts}%
\usepackage{mathtools}
\usepackage{amssymb}%
\usepackage[left=1.5in, right=1.5in]{geometry}
\usepackage{amsthm}
\usepackage{todonotes}
\usepackage{tikz}
\usepackage{pgfplots}
\usetikzlibrary{arrows.meta}
\usetikzlibrary{matrix}
\usepackage{enumitem}
\usepackage{comment}

\renewcommand{\tilde}{\widetilde}

\newcommand{\mf}{\mathfrak}

\newcommand\precdot{\mathrel{\ooalign{$\prec$\cr
  \hidewidth\raise0.000ex\hbox{$\cdot\mkern0.5mu$}\cr}}}

\newcommand{\Z}{\mathbb{Z}}

\newcommand{\C}{\mathbb{C}}
\newcommand{\ph}{\varphi}
\renewcommand{\S}{\mathfrak{S}}

\DeclareMathOperator{\inv}{Inv}
\DeclareMathOperator{\End}{End}
\DeclareMathOperator{\spn}{span}
\DeclareMathOperator{\wt}{wt}

\newtheorem{theorem}{Theorem}[section]
\newtheorem{prop}[theorem]{Proposition}

\newtheorem{lem}[theorem]{Lemma}
\newtheorem{cor}[theorem]{Corollary}

\theoremstyle{definition}

\theoremstyle{remark}
\newtheorem{remark}[theorem]{Remark}

\begin{document}

\title{The Sperner property for $132$-avoiding intervals in the weak order}
\author{Christian Gaetz}
\address{Department of Mathematics, Massachusetts Institute of Technology, Cambridge, MA.}
\email{\href{mailto:gaetz@mit.edu}{gaetz@mit.edu}} 

\author{Katherine Tung}
\address{Menlo School, Atherton, CA.}
\email{\href{mailto:KatherineATung@gmail.com}{KatherineATung@gmail.com}}

\thanks{C.G. was partially supported by an NSF Graduate Research Fellowship under grant No. 1122374}

\date{\today}

\begin{abstract}
A well-known result of Stanley from 1980 implies that the weak order on a maximal parabolic quotient of the symmetric group $S_n$ has the Sperner property; this same property was recently established for the weak order on all of $S_n$ by Gaetz and Gao, resolving a long-open problem.  In this paper we interpolate between these results by showing that the weak order on any parabolic quotient of $S_n$ (and more generally on any $132$-avoiding interval) has the Sperner property.

This result is proven by exhibiting an action of $\mf{sl}_2$ respecting the weak order on these intervals.  As a corollary we obtain a new formula for principal specializations of Schubert polynomials. Our formula can be seen as a strong Bruhat order analogue of Macdonald's reduced word formula. This proof technique and formula generalize work of Hamaker--Pechenik--Speyer--Weigandt and Gaetz--Gao.  
\end{abstract}


\maketitle

\section{Introduction} \label{sec:intro}

A ranked poset $P$ with rank decomposition $P=P_0 \sqcup \cdots \sqcup P_r$ is \emph{$k$-Sperner} if no union of $k$ antichains of $P$ is larger than the union of the $k$ largest $P_i$'s, and is \emph{strongly Sperner} if it is $k$-Sperner for $k=1,2,\ldots$.  The $1$-Sperner property is often simply called the \emph{Sperner property}.

The reader is referred to Section~\ref{sec:background} for background and definitions.  The following well-known result of Stanley \cite{Stanley-hard-lefschetz} was used to resolve an open problem of Erd\H{o}s:

\begin{theorem}[Stanley \cite{Stanley-hard-lefschetz}]
\label{thm:stanley-sperner}
The strong Bruhat order on any parabolic quotient $S_n^J$ of the symmetric group is strongly Sperner.
\end{theorem}

The weak (Bruhat) order is a natural coarsening of the strong Bruhat order which has the same rank decomposition.  Establishing the Sperner property for the weak order on a parabolic quotient would thus be a strengthening of Theorem~\ref{thm:stanley-sperner}.  Bj\"{o}rner \cite{Bjorner} asked whether the weak order on the whole symmetric group $S_n=S_n^{\emptyset}$ had the (strong) Sperner property; it was conjectured by Stanley \cite{Schubert-shenanigans} that it does, and this was proven in \cite{sperner-paper} by establishing a certain action of the Lie algebra $\mf{sl}_2$ which respects the weak order.  This technique was developed further in \cite{duality-paper, path-counting, Hamaker} in order to establish new formulas for principal evaluations of Schubert polynomials.  In this paper we generalize this $\mf{sl}_2$ action (Theorem~\ref{thm:sl2}) in order to establish the strong Sperner property (Corollary~\ref{cor:sperner}) for all $132$-avoiding intervals (i.e., intervals in the weak order below $132$-avoiding permutations), and in particular on all parabolic quotients $S_n^J$ of the symmetric group.  This approach yields new, more general, formulas for principal evaluations of Schubert polynomials (Theorem~\ref{thm:schubert}).

\subsection{A generalized $\mf{sl}_2$-action}
Given a set $X$ we let $\C X$ denote the vector space of formal linear combinations of elements of $X$. Let $\ell(\sigma)$ denote the Coxeter length of a permutation $\sigma$.  Let $\leq$ denote the right weak order on $S_n$ and $\preceq$ the strong Bruhat order.  For $\pi \in S_n$ a $132$-avoiding permutation, we define operators $E,F,H \in \End(\C [e,\pi]_R)$ by 
\begin{align*}
    E\sigma &= \sum_{\sigma \precdot \sigma t_{ij} \leq \pi} \wt^{\pi}\left(\sigma,\sigma t_{ij}\right) \sigma t_{ij}, \\
    F\sigma &= \sum_{\sigma s_i \lessdot \sigma \leq \pi} i \:  \sigma s_i, \\
    H\sigma &= \left(2 \ell(\sigma)-\ell(\pi) \right) \sigma,
\end{align*}
where 
\begin{align}
\label{eq:def-of-wt}
\wt^{\pi}(\sigma,\sigma t_{ij}) \coloneqq 1&+|\{k>j \: | \: \sigma_i < \sigma_k < \sigma_j \}| \\ \nonumber &+ |\{k>j \: | \: \pi^{-1}(\sigma_j) < \pi^{-1}(\sigma_k) < \pi^{-1}(\sigma_i)\}|, 
\end{align}
and where $[e,\pi]_R$ denotes the interval below $\pi$ in the right weak order.  The operator $F$ is the restriction of an operator suggested by Stanley \cite{Schubert-shenanigans} and $E$ is a significant generalization of the operator used in \cite{sperner-paper}.  Let $e,f,h$ denote the standard generators for the Lie algebra $\mf{sl}_2(\C)$ (see Section~\ref{sec:background}).

\begin{theorem}
\label{thm:sl2}
Let $\pi$ be a $132$-avoiding permutation, then the map sending $e \mapsto E, f \mapsto F, h \mapsto H$ defines a representation of $\mf{sl}_2(\C)$ on $\C [e,\pi]_R$.
\end{theorem}

By observations of Stanley \cite{Stanley-hard-lefschetz} and Proctor \cite{Proctor2}, Theorem~\ref{thm:sl2} implies the following corollary:

\begin{cor}
\label{cor:sperner}
The weak order interval $[e,\pi]_R$ is strongly Sperner whenever $\pi$ avoids the pattern $132$. In particular, the weak order on any parabolic quotient $S_n^J$ is strongly Sperner.  
\end{cor}

\begin{remark}
When $S_n^J$ is a \emph{maximal} parabolic quotient, the weak order and strong order agree, so that Theorem~\ref{thm:stanley-sperner} implies the weak order result in this case.  As \cite{sperner-paper} established the Sperner property for the \emph{minimal} parabolic quotient $S_n=S_n^{\emptyset}$, Corollary~\ref{cor:sperner} can also be seen as interpolating between these two results.
\end{remark}

\subsection{Principal specializations of Schubert polynomials}
See Section~\ref{sec:background-schubert} for the definition of the Schubert polynomials $\S_\sigma(x_1,\ldots,x_{n-1})$.  These polynomials are a central object of study in algebraic combinatorics and combinatorial algebraic geometry, and their \emph{principal specializations} $\S_\sigma(1,\ldots,1)$ have received considerable study \cite{Macdonald, Merzon-Smirnov, Morales-Pak-Panova}.  

Given a saturated chain $C: x_1 \lessdot \cdots \lessdot x_k$ in a poset $P$ and a weight function $\wt$ on the covering relations of $P$, we extend the weight function to $C$ multiplicatively:
\[
\wt(C)=\prod_{i=1}^{k-1} \wt(x_i \lessdot x_{i+1}).
\]

We establish the following formula for $\S_\sigma(1,\ldots,1)$, generalizing the main result of \cite{duality-paper}:

\begin{theorem}
\label{thm:schubert} 
Let $\sigma \in S_n$, and choose any $132$-avoiding permutation $\pi$ which is greater than $\sigma$ in the right weak order (such a $\pi$ always exists, since $w_0$ is $132$-avoiding).  Then
\[
\S_\sigma(1,\ldots,1)=\frac{1}{(\ell(\pi)-\ell(\sigma))!}\sum_{C:\sigma \to \pi} \wt^{\pi}(C),
\]
where the sum is over all saturated chains from $\sigma$ to $\pi$ in the strong Bruhat order on $[e,\pi]_R$, and where $\wt^{\pi}$ is defined by (\ref{eq:def-of-wt}).
\end{theorem}

The freedom to choose $\pi$ in Theorem~\ref{thm:schubert} is a unique feature of our result.  This is an advantage over previously known formulas for $\S_\sigma(1,\ldots,1)$ such as Macdonald's reduced word formula (Theorem~\ref{thm:macdonald-identity}) and its various strong order analogues (studied in \cite{duality-paper, path-counting}) in which one must always count weighted chains from $e$ to $\sigma$ or from $\sigma$ to $w_0$.  For example, Theorem~\ref{thm:schubert} makes it clear that $\S_\sigma(1,\ldots,1)=1$ when $\sigma$ is $132$-avoiding, as we may choose $\pi=\sigma$; this fact is not apparent from the other formulas.

\subsection{Outline}

The remainder of the paper is organized as follows: Section~\ref{sec:background} contains background, notation, and conventions for the weak and strong Bruhat orders, Schubert polynomials, and $\mf{sl}_2$-representations.  Section~\ref{sec:sl2} proves Theorem~\ref{thm:sl2} and Corollary~\ref{cor:sperner}.  Finally, Section~\ref{sec:schubert} derives Theorem~\ref{thm:schubert}.

\section{Background and conventions} \label{sec:background}

\subsection{The weak and strong Bruhat orders}
\label{sec:background-orders}
The material in this section may be found, for example, in \cite{bjorner-brenti}.

For $i=1,\ldots, n-1$ we write $s_i$ for the simple transposition $(i \text{ } i+1)$ in the symmetric group $S_n$.  Given any element $\sigma \in S_n$, the \emph{length} $\ell(\sigma)$ of $\sigma$ is the smallest number $\ell$ such that $\sigma=s_{i_1} \cdots s_{i_{\ell}}$ for some indices $i_j \in [n-1]\coloneqq \{1,2,\ldots,n-1\}$.  Such an expression of minimal length is called a \emph{reduced word} for $\sigma$.  The \emph{(right) weak (Bruhat) order} is the partial order $(S_n, \leq)$ having a cover relation $\sigma \lessdot \sigma s_i$ whenever $\ell(\sigma s_i)=\ell(\sigma )+1$.  We write $[x,y]_R$ for the interval $\{z \: | \: x \leq z \leq y\}$ in right weak order.

The \emph{inversion set} $\inv(\sigma)$ of a permutation $\sigma \in S_n$ is defined by:
\[
\inv(\sigma)=\{(\sigma_i, \sigma_j) \: | \: 1 \leq i < j \leq n, \: \sigma_i > \sigma_j \},
\]
where $\sigma_1 \ldots \sigma_n$ is the one-line notation for $\sigma$.  It is well-known that $\ell(\sigma)=|\inv(\sigma)|$ and that the weak order is characterized by containment of inversion sets:
\begin{prop}
\label{prop:weak-characterized-by-inversions}
For $u,v \in S_n$ we have $u \leq v$ if and only if $\inv(u) \subseteq \inv(v)$.
\end{prop}

We sometimes write $t_{ij}$ for the transposition $(i \: j)$.  The \emph{strong (Bruhat) order} (often referred to in the literature just as the \emph{Bruhat order}) is the partial order $(S_n, \preceq)$ having a cover relation $\sigma \precdot \sigma t_{ij}$ whenever $\ell(\sigma t_{ij})=\ell(\sigma)+1$. Equivalently, $\sigma \precdot \sigma t_{ij}$ whenever $\sigma_i<\sigma_j$ and there is no index $k$ with $i<k<j$ such that $\sigma_i<\sigma_k<\sigma_j$.  Thus $\preceq$ is a refinement of $\leq$ and the two partial orders share the same rank structure, with rank function $\ell$.  Both the weak order and the strong order have unique minimal element $e=12\ldots n$, the identity permutation, and unique maximal element $w_0=n(n-1)\ldots 1$, called the \emph{longest element}, which has length ${n \choose 2}$. 

For $J \subseteq \{s_1,\ldots,s_{n-1}\}$, the corresponding \emph{parabolic quotient} is 
\[
S_n^J=\{\sigma \in S_n \: | \: s_i \in J \implies \sigma_i < \sigma_{i+1}\}. 
\]
Each parabolic quotient has a unique element $w_0^J$ of maximum length, and the set $S_n^J$ coincides with the interval $[e,w_0^J]_L$ in the \emph{left} weak order.  This interval is easily seen to be isomorphic as a poset to the interval $[e,w_0^J]_R$ in right weak order, so for the purposes of investigating the Sperner property it suffices to deal only with the right weak order throughout the paper.  The elements $w_0^J$ are all $132$-avoiding (see Section~\ref{sec:background-schubert} for the definition of pattern avoidance), so Corollary~\ref{cor:sperner} applies to the weak order on any parabolic quotient.  There are, however, many more $132$-avoiding permutations to which Theorem~\ref{cor:sperner} applies that are not of the form $w_0^J$.

\subsection{Schubert polynomials and pattern avoidance}
\label{sec:background-schubert}

See \cite{Macdonald} or \cite{Manivel} for background on Schubert polynomials.

The symmetric group $S_n$ acts on the polynomial algebra $\C[x_1,\ldots, x_n]$ by permutation of variables.  For each $i=1,\ldots,n-1$ the \emph{$i$-th Newton divided difference operator} $N_i$ acts on polynomials by
\[
N_i f = \frac{f- s_if}{x_i-x_{i+1}}.
\]
Given an $n$-tuple $\alpha$, we write $x^{\alpha}$ for the monomial $x_1^{\alpha_1}\cdots x_n^{\alpha_n}$; we write $\alpha \leq \beta$ if $\alpha_i \leq \beta_i$ for all $i$ and let $\rho_n$ denote the staircase $(n-1,n-2,\ldots,2,1,0)$.

The \emph{Schubert polynomials} $\S_\sigma$ are indexed by permutations $\sigma \in S_n$ and may be defined recursively as follows: 
\begin{itemize}
    \item $\S_{w_0}=x^{\rho_n}$, and 
    \item $\S_{s_i\sigma}=N_i \S_\sigma $ if $\ell(s_i\sigma)=\ell(\sigma)-1$.
\end{itemize}
The polynomial $\S_\sigma$ is homogeneous of degree $\ell(\sigma)$ and these polynomials form a basis for the space
\[
U_n = \spn_{\C}\{x^{\alpha} \: | \: \alpha \leq \rho_n\}
\]
as $\sigma$ ranges over the symmetric group $S_n$.  Schubert polynomials are of great interest because they represent cohomology classes of Schubert varieties in the cohomology $H^*(G/B)$ of the complete flag variety, but this perspective will not be important here.

\begin{remark}
Our convention for the indexing of Schubert polynomials differs from that often used in the literature by an inverse in the subscript. However, our convention agrees with that used in \cite{duality-paper, path-counting, Hamaker}.  
\end{remark}

We say a permutation $\sigma \in S_n$ \emph{contains the pattern} $\pi \in S_k$ if there are indices $1 \leq i_1 < \cdots < i_k \leq n$ such that $\sigma_{i_1}, \ldots, \sigma_{i_k}$ are in the same relative order as $\pi_1,\ldots, \pi_k$.  If $\sigma$ does not contain the pattern $\pi$, it is said to \emph{avoid} $\pi$.  The following proposition is well known:

\begin{prop}[See \cite{Macdonald}]
\label{prop:schubert-is-code}
The Schubert polynomial $\S_\sigma$ consists of a single monomial if and only if $\sigma$ avoids $132$.
\end{prop}

\subsection{Facts about $\mf{sl}_2$-representations}
\label{sec:background-sl2}
See \cite{Kirillov} for definitions and basic facts regarding Lie algebras and their representations.

The Lie algebra $\mf{sl}_2=\mf{sl}_2(\C)$ has standard generators
\begin{align*}
    &e=\begin{pmatrix} 0 & 1 \\ 0 & 0 \end{pmatrix},& &f=\begin{pmatrix} 0 & 0 \\ 1 & 0 \end{pmatrix},&
    &h=\begin{pmatrix} 1 & 0 \\ 0 & -1 \end{pmatrix},&
\end{align*}
which satisfy the defining relations
\begin{align*}
    &[e,f]=h,& &[h,e]=2e,& &[h,f]=-2f.&
\end{align*}
We write $\mf{b}^-$ for the subalgebra of $\mf{sl}_2$ generated by $h$ and $f$.  If $V$ is an $\mf{sl}_2$-representation (assumed throughout to be complex and finite-dimensional), we let
\[
V[k]=\{v \in V \: | \: hv = kv \}
\]
denote the $k$-weight space.  It is well-known that $V[k] \neq 0$ only if $k \in \Z$ and that $eV[k] \subseteq V[k+2]$ and $fV[k] \subseteq V[k-2]$.  There is a unique irreducible $\mf{sl}_2$-representation $V_d$ of each dimension $d=1,2,\ldots$.  The representation $V_d$ has weights $d-1, d-3, \ldots, -(d-1)$, with each weight space 1-dimensional, a basis for $V_d$ being given by $v, fv, \ldots, f^{d-1}v$ where $v \in V_d[d-1]$ is any nonzero vector.

The following theorem follows from the Jacobson--Morozov Theorem \cite{Jacobson, Morozov} and was shown explicitly by Proctor \cite{Proctor2}: 

\begin{theorem}[Jacobson \cite{Jacobson}, Morozov \cite{Morozov}, Proctor \cite{Proctor2}]
\label{thm:at-most-one-action}
Let $V$ be a finite-dimensional complex vector space.  Given $E,H$ (resp. $F,H$) in $\End(V)$, there is at most one element $F$ (resp. $E$) in $\End(V)$ such that $(e,f,h)\mapsto (E,F,H)$ determines an $\mf{sl}_2$-representation.
\end{theorem}

For an $\mf{sl}_2$-representation $V$ and $V' \subseteq V$ a subspace (not necessarily a subrepresentation) we say $V'$ is \emph{weight symmetric} if 
\[
\dim(V' \cap V[k])=\dim(V' \cap V[-k])
\]
for all weights $k$.  The following elementary proposition will be useful in Section~\ref{sec:schubert}:

\begin{prop}
\label{prop:b-submodule}
Let $V$ be an $\mf{sl}_2$-representation and $V' \subseteq V$ a subspace, closed under the action of $\mf{b}^-$, which is weight symmetric.  Then $V'$ is closed under the action $\mf{sl}_2$ (that is, $V'$ is an $\mf{sl}_2$-subrepresentation of $V$).
\end{prop}
\begin{proof}
Write $V'[k]$ for $V' \cap V[k]$.  Let $k_{max}$ be the largest weight such that $V'[k_{max}] \neq 0$ and let $v_1,\ldots,v_r$ be a basis for $V'[k_{max}]$.  It is a standard fact about $\mf{sl}_2$-representations that the linear map $f^{k}:V[k] \to V[-k]$ is an isomorphism; thus $\{f^{p}v_1, \ldots, f^{p}v_r\}$ is linearly independent for each $p=0,1,\ldots,k_{max}$.  By the weight symmetry of $V'$ and the maximality of $k_{max}$ we see that $V'[k]=0$ for $k<-k_{max}$, so $f^{k_{max}+1}v_i=0$ for all $i$.

Suppose that $ev_i \neq 0$, and let $\ell$ be the largest number so that $e^{\ell}v_i \neq 0$.  Then $e^{\ell}v_i$ is a highest weight vector of weight $k_{\max}+2\ell$, so $f^{k_{max}+2\ell}e^{\ell}v_i \neq 0$.  But repeated applications of the relation $[e,f]=h$ show that $f^{\ell}e^{\ell}v_i$ is a scalar multiple of $v_i$ and the previous paragraph then implies that $f^{k_{max}+1+\ell}e^{\ell}v_i=0$, a contradiction, thus $ev_i=0$.  It now follows that 
\[
V''=\spn_{\C}\{f^j v_i \: | \: i=1,\ldots, r, \: j=0,\ldots, k_{max} \} \subseteq V'
\]
is closed under the action of $\mf{sl}_2$.  Repeating this argument with $V'/V'' \subseteq V/V''$ rather than $V' \subseteq V$ completes the proof (this process clearly terminates since $V$ is finite-dimensional).
\end{proof}

\section{An $\mf{sl}_2$-action for weak order intervals} \label{sec:sl2}

Throughout this section we let $\pi \in S_n$ be a $132$-avoiding permutation and let $E,F,H \in \End(\C[e,\pi]_R)$ denote the associated linear operators defined in Section~\ref{sec:intro}.  

In order to prove Theorem~\ref{thm:sl2} we need to verify the following the relations:
\begin{align*}
    [H,E]&=2E, \\
    [H,F]&=-2F, \\
    [E,F]&=H,
\end{align*}
where $[X,Y]=XY-YX$ denotes the commutator.  The first two of these are automatic, since $E,F$ raise and lower length by one, respectively. We now focus on proving that $[E,F]=H$.  We view $[E,F]=EF-FE$ as a matrix with rows and columns indexed by the elements of $[e,\pi]_R$.  Propositions~\ref{prop:diag} and \ref{prop:off-diag} below together imply the desired relation and thus Theorem~\ref{thm:sl2}; the proofs of these propositions appear in Sections~\ref{sec:proof-of-diag} and \ref{sec:proof-of-off-diag}.

\begin{prop}
\label{prop:diag}
For any $\sigma \in [e,\pi]_R$ we have $(EF-FE)_{\sigma,\sigma}=2\ell(\sigma)-\ell(\pi)$.
\end{prop}

\begin{prop}
\label{prop:off-diag}
For $\sigma,\tau \in [e,\pi]_R$ with $\sigma\ne \tau$, $(E F -FE)_{\sigma,\tau} = 0$.
\end{prop}

Assuming these propositions, and thus Theorem~\ref{thm:sl2}, have been established, we now obtain Corollary~\ref{cor:sperner}, which states that $[e,\pi]_R$ is strongly Sperner: 

\begin{proof}[Proof of Corollary~\ref{cor:sperner}]
One of the fundamental observations in \cite{Stanley-hard-lefschetz} is that the strong Sperner property of a ranked poset $P=P_0 \sqcup \cdots \sqcup P_r$ is implied by the existence of a linear operator $\ph \in \End(\C P)$ sending elements $x \in P$ to linear combinations of elements covered by $x$, such that $\ph^{r-2i}:\C P_{r-i} \to \C P_i$ is an isomorphism for each $i$.  

It is a standard fact (see, e.g. \cite{Kirillov}) that for any $\mf{sl}_2$-representation $V$ the linear map $f^k: V[k] \to V[-k]$ is an isomorphism for all $k$.  Thus by Theorem~\ref{thm:sl2} the operator $F$ has the desired properties with respect to the poset $([e,\pi]_R, \leq)$.
\end{proof}

\subsection{Proof of Proposition \ref{prop:diag}}
\label{sec:proof-of-diag}

\begin{proof}[Proof of Proposition \ref{prop:diag}]

We will proceed by comparing the general case to the case $\pi=w_0$, which was established in \cite{sperner-paper}. We write $E$ for the raising operator associated to our fixed $132$-avoiding permutation $\pi$ and write $E_{w_0}$ when we mean to refer to the operator from \cite{sperner-paper}, where $\pi=w_0$.  The operator $F$ does not depend on $\pi$, except in determining its domain $\C [e,\pi]_R$; since $\C [e,\pi]_R$ is naturally a subspace of $\C S_n$, this should cause no ambiguity, and we do not distinguish notationally between the operators on $\C [e,\pi]_R$ and on $\C[e,w_0]_R = \C S_n$.

Gaetz and Gao \cite{sperner-paper} showed:
\[
(E_{w_0}F-FE_{w_0})_{\sigma,\sigma} =2\ell(\sigma)-\ell(w_0) = 2\ell(\sigma) -{n \choose 2}. 
\]
Subtracting, the desired equality in Proposition~\ref{prop:diag} is equivalent to the following: 
\begin{align}
\label{eq:difference}
    (E F-F E)_{\sigma,\sigma}-(E_{w_0}F-FE_{w_0})_{\sigma,\sigma} &= 2\cdot\ell(\sigma)-\ell(\pi)-\Bigg(2\cdot\ell(\sigma)-{n \choose 2}\Bigg)  \\
    \nonumber &= {n \choose 2}-\ell(\pi).
\end{align}
Note that the right-hand side does not depend on $\sigma$.  Fix for the remainder of this section a permutation $\sigma \in [e,\pi]_R$ and let $D$ denote the difference 
\[
(E F-FE)_{\sigma,\sigma}-(E_{w_0}F-FE_{w_0})_{\sigma,\sigma}
\]
on the left-hand side. We now show $D={n \choose 2}-\ell(\pi)$ using a combinatorial argument.

For permutations $\rho,\tau$ with $\rho\precdot\tau$ in the strong order, we say the \textit{up weight} of the edge between $\rho$ and $\tau$ is the coefficient of $\tau$ in $E\rho$, namely $\wt^{\pi}(\rho,\tau)$. The \textit{down weight} of the edge between $\rho$ and $\tau$ is the coefficient of $\rho$ in $F\tau$, namely $i$ if $\tau=\rho s_i$ and 0 if $\tau$ cannot be expressed in this form. 

For $1\le i\le n-1$, we say the swap $s_i$ is \textit{forbidden} if $\sigma s_i\not\in[e,\pi]_R,$ and otherwise is \emph{allowed}. Although $E$ respects the strong order, not the weak order, we only need to consider the weak order edges incident to $\sigma$ in order to compute $(E F)_{\sigma,\sigma}$ or $(FE)_{\sigma,\sigma}$, since the up weights on other edges are multiplied by the down weight of $0$.

Define the following sets used in computing the up weights:
\[
A_i(\sigma)=A_i=\{\sigma_k ~ \vert ~ k>i+1,\min{(\sigma_i,\sigma_{i+1})}<\sigma_k<\max{(\sigma_i,\sigma_{i+1})}\} 
\]
\begin{align*}
    B_i(\sigma,\pi)=B_i=\{\sigma_k ~\vert~ k>i+1,\min{(\pi^{-1}\sigma_i,\pi^{-1}\sigma_{i+1})}<&\pi^{-1}\sigma_k \\ <&\max{(\pi^{-1}\sigma_i,\pi^{-1}\sigma_{i+1})}\}.
\end{align*}
Note that $A_i(\sigma)=A_i(\sigma s_i)$ and $B_i(\sigma,\pi)=B_i(\sigma s_i,\pi)$.

\begin{lem}
\label{lem:forbidden-label}
Let $\sigma \in[e,\pi]_R$ and suppose $s_i$ is forbidden. Then $\ell(\sigma s_i) = \ell(\sigma)+1$ and $A_i=B_i=\emptyset.$ 
\end{lem}
\begin{proof}
Since $\sigma \le \pi$, if $\sigma s_i < \sigma$ we would have $\sigma s_i \in [e,\pi]_R$; but by assumption, $s_i$ is forbidden, so we must have $\sigma < \sigma s_i$ and therefore $\ell(\sigma s_i)=\ell(\sigma)+1$.  We have $\sigma_i < \sigma_{i+1}$ so $\sigma$ does not have $(\sigma_i,\sigma_{i+1})$ as an inversion; by Proposition~\ref{prop:weak-characterized-by-inversions} and the fact that $s_i$ is forbidden, $(\sigma_{i+1},\sigma_i)$ is not an inversion of $\pi$.

We want to show that $A_i$ is empty. Suppose by way of contradiction that the value $\sigma_k$ is between $\sigma_i$ and $\sigma_{i+1}$ for some $k > i+1$. Then $\sigma$ inverts the values $\sigma_{i+1}$ and $\sigma_k$, and $\pi>\sigma$ so $\pi$ must also invert the values $\sigma_{i+1}$ and $\sigma_k$ by Proposition~\ref{prop:weak-characterized-by-inversions}. But this implies that the values $\sigma_i, \sigma_{i+1}, \sigma_k$ form an occurrence of the pattern $132$ in $\pi$, a contradiction, since $\pi$ is assumed to avoid $132$. Thus $A_i=\emptyset$. 

To show $B_i$ is empty, suppose by way of contradiction for some $k > i+1$ that $\pi^{-1} \sigma_k$ is between $\pi^{-1} \sigma_i$ and $\pi^{-1} \sigma_{i+1}.$ Since $\sigma_k$ appears after $\sigma_{i+1}$ in $\sigma$ but before it in $\pi > \sigma$, it must be that $\sigma_{i+1} < \sigma_k.$ Then the values $\sigma_i, \sigma_k, \sigma_{i+1}$ form a $132$ pattern in $\pi$, a contradiction, so $B_i=\emptyset.$
\end{proof}

If the swap $s_i$ is allowed for $\sigma \leq \pi$, then either $\sigma \lessdot \sigma s_i$, so $s_i$ corresponds to an upward edge from $\sigma$ and this edge contributes to $(-FE )_{\sigma,\sigma}$ and $(-FE_{w_0})_{\sigma,\sigma}$, or else $\sigma s_i \lessdot \sigma$ and $s_i$ is a downward edge contributing to $(E F)_{\sigma,\sigma}$ and $(E_{w_0}F)_{\sigma,\sigma}$). Regardless of the direction, the down weight of this edge is $i$. The up weight in $E_{w_0}$ is $1+2|A_i|$ and $1+|A_i|+|B_i|$ in $E$.  Therefore
\begin{equation}
\label{eq:D-as-sum}
    D  = \sum_{i=1}^{n-1} \varepsilon_i,
\end{equation}
where
\[
\varepsilon_i=\begin{cases} i, &\text{$s_i$ forbidden,} \\ i(|A_i| - |B_i|), &\text{$s_i$ allowed and $\sigma \lessdot \sigma s_i$,} \\ i(|B_i|-|A_i|), &\text{$s_i$ allowed and $\sigma s_i \lessdot \sigma$.}\end{cases}
\]


We illustrate this information in a structure called a \textit{sign grid} with rows labeled $1,2,\ldots,n-1$ and columns labeled $0,1,\ldots,n$ (see Figure~\ref{fig:sign-grid}). We write $S_{ij}$ for the entry in the $i$-th row and $j$-th column, these are $+1$, $-1$, or $0$ (empty) as follows:
\begin{itemize}
    \item The cell $S_{i0}$ is $+1$ if swap $s_i$ is forbidden, and $0$ if $s_i$ is allowed. 
    \item For $1\le j\le n$, $S_{ij}=\lambda_i \mu_j$, where 
\begin{align*}
&\lambda_i=\begin{cases} +1, &\text{if $\sigma \lessdot \sigma s_i$} \\ -1, &\text{if $\sigma s_i \lessdot \sigma$} \end{cases}
&\mu_j=\begin{cases} +1, &\text{if $j\in A_i(\sigma)\setminus B_i(\sigma)$} \\ -1, &\text{if $j\in B_i(\sigma)\setminus A_i(\sigma)$} \end{cases}
\end{align*}
    \item All other entries are $0$.
\end{itemize}   

The formula (\ref{eq:D-as-sum}) can now be rewritten as 
\begin{equation}
    \label{eq:D-from-grid}
    D = \sum_{i,j} i S_{ij}.
\end{equation}
To compute $D$, we will first sum by columns. The desired value ${n \choose 2}-\ell(\pi)$ for $D$ is the number of non-inverted pairs in $\pi$. These pairs can be grouped based on the larger number in each pair. We claim that for each positive $j$, the sum of the $j$-th column $\sum_i iS_{ij}$ is the number of non-inverted pairs with $j$ the larger number in each pair, minus a correction term coming from the entries in column $0$. 

To prove these claims, we consider a \textit{permutation path} within $[n] \times [n]$ for each column $j$, constructed as follows: Let $k=\sigma^{-1}j$ and for $1\le i\le k-1$, draw a vector from point $(\sigma_i,\pi^{-1}\sigma_i)$ to $(\sigma_{i+1},\pi^{-1}\sigma_{i+1})$. We also consider the lines $x=\sigma_k$ and $y=\pi^{-1}\sigma_k$, dividing the plane into four quadrants numbered counterclockwise from the top right (see Figure~\ref{fig:permutation-path}).

\begin{figure}
\tikzset{ 
    table/.style={
        matrix of nodes,
        row sep=-\pgflinewidth,
        column sep=-\pgflinewidth,
        nodes={
            rectangle,
            draw=black,
            align=center
        },
        minimum height=1.5em,
        text depth=0.5ex,
        text height=2ex,
        nodes in empty cells,
        every even row/.style={
            nodes={fill=gray!20}
        },
        column 1/.style={
            nodes={text width=2em,font=\bfseries}
        },
        row 1/.style={
            nodes={
                fill=black,
                text=white,
                font=\bfseries
            }
        }
    }
}
\begin{tikzpicture}
\matrix (first) [table,text width=1.5em]
{
  & 0 & 1 & 2 & 3 & 4 & 5 & 6 & 7 & 8 \\
1 &   &   &   &   &   &   &   &   &    \\
2 &   &   &   &   & +1 &  & -1 & -1 &    \\
3 & +1 &   &   &   &    &   &   &   &    \\
4 &    &   &   &   &   &   &   &  +1 &      \\
5 &   &   &   &   &   &   &   &   &     \\
6 &   &   &   &   &   &   &   &   &    \\
7 & +1 &   &   &   &   &   &   &   &    \\
};
\end{tikzpicture}
\caption{Sign grid for $\pi=56732418, ~\sigma = 32564178$. }
\label{fig:sign-grid}
\end{figure}

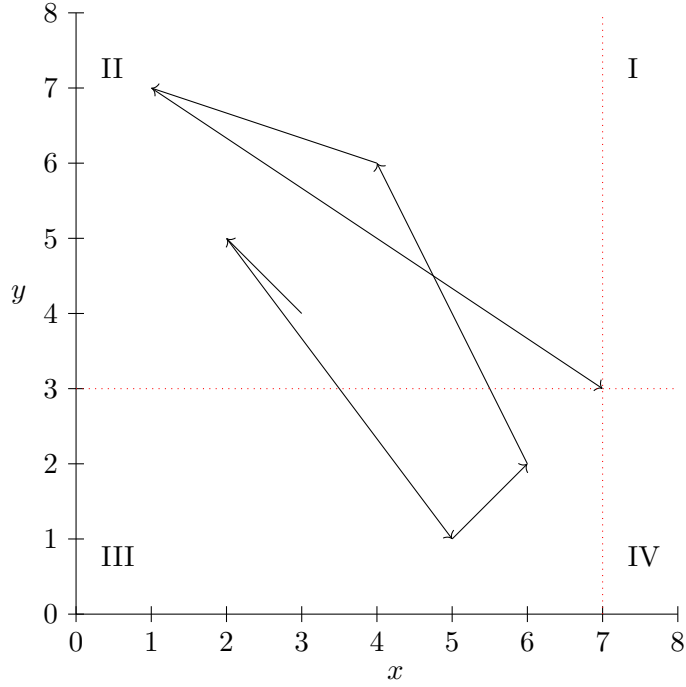
\begin{figure}
\begin{tikzpicture}[scale=1]
\draw   (0,0) -- (8,0)
        (0,0) -- (0,8);
\foreach \x in {0,...,8}
     \draw (\x,3pt) -- (\x,-3pt)
        node[anchor=north] {\x};
\foreach \y in {0,...,8}
     \draw (3pt,\y) -- (-3pt,\y)
        node[anchor=east] {\y};
\draw[->](3,4) -- (2,5);
\draw[->](2,5) -- (5,1);
\draw[->](5,1) -- (6,2);
\draw[->](6,2) -- (4,6);
\draw[->](4,6) -- (1,7);
\draw[->](1,7) -- (7,3);
\draw[dotted,red](0,3) -- (8,3);
\draw[dotted,red](7,0) -- (7,8);
\node[above right=0pt of {(7.2,7)}] {I};
\node[above right=0pt of {(0.2,7)}] {II};
\node[above right=0pt of {(0.2,0.5)}] {III};
\node[above right=0pt of {(7.2,0.5)}] {IV};
\node[above right=0pt of {(-1,4)}] {$y$};
\node[above right=0pt of {(4,-1)}] {$x$};
\end{tikzpicture}
\caption{Permutation path for $\pi=56732418$, $\sigma=32564178,$ and  $c=7=\sigma_7$. Dotted lines $x=\sigma_7,y=\pi^{-1}\sigma_7$ are shown, dividing the region into Quadrants I-IV.}
\label{fig:permutation-path}
\end{figure}

We prove a few properties of the permutation path relating to these quadrants, helping us explain some patterns in the sign grid.

\begin{lem}
\label{lem:perm-path-facts} 
Let the permutation path and quadrants be constructed as above, with respect to $k=\sigma^{-1} j$.
\begin{itemize}
    \item[(a)] Steps in the permutation path are left-to-right if the corresponding swap is length-increasing, and right-to-left if the corresponding swap is length-decreasing. The edge $(\sigma_i,\pi^{-1}\sigma_i) \to (\sigma_{i+1},\pi^{-1}\sigma_{i+1})$ crosses the line $x=\sigma_k$ when $k \in A_i$, and it crosses $y=\pi^{-1}\sigma_k$ when $k \in B_i$. 
    \item[(b)] No point in the permutation path is in the interior of Quadrant I.
    \item[(c)] No vector in the permutation path points down and to the left.
    \item[(d)] A vector points up and to the right if and only if it corresponds to a forbidden swap. If a vector corresponds to a forbidden swap, the vector does not change quadrants. 
    \item[(e)] A vector $(\sigma_i,\pi^{-1}\sigma_i)\to(\sigma_{i+1},\pi^{-1}\sigma_{i+1})$ crossing from Quadrant II to III must point down-right and $S_{ij}=-1$. Likewise, a vector $(\sigma_i,\pi^{-1}\sigma_i)\to(\sigma_{i+1},\pi^{-1}\sigma_{i+1})$ crossing from Quadrant III to II must point up-left and $S_{ij}=1$. 
    \item[(f)] A vector $(\sigma_i,\pi^{-1}\sigma_i)\to(\sigma_{i+1},\pi^{-1}\sigma_{i+1})$ crossing from Quadrant IV to III must point up-left and $S_{ij}=-1$. Likewise, a vector $(\sigma_i,\pi^{-1}\sigma_i)\to(\sigma_{i+1},\pi^{-1}\sigma_{i+1})$ crossing from Quadrant III to IV must point down-right and $S_{ij}=1$.
    \item[(g)] If a vector starts and ends in the same quadrant, or goes between even quadrants, the corresponding sign grid entry $S_{ij}$ is $0$.
\end{itemize}
\end{lem}
\begin{proof}
\text{}
\begin{itemize}
    \item[(a)] The $x$-coordinate increases if $\sigma_i < \sigma_{i+1}$, in which case $\sigma < \sigma s_i$. The edge crosses $x = \sigma_k$ when $\sigma_k$ is between $\sigma_i$ and $\sigma_{i+1}$, which means $k \in A_i.$ Similarly, the edge crosses $y=\pi^{-1} \sigma_k$ when $\pi^{-1}\sigma_k$ is between $\pi^{-1}\sigma_i$ and $\pi^{-1}\sigma_{i+1}$.
    \item[(b)] Suppose there exists some $(\sigma_i,\pi^{-1}\sigma_{i})$ in Quadrant I. Then $i<k$ and $\sigma_k<\sigma_i$, so $\sigma$ inverts the values $\sigma_i$ and $\sigma_k$. On the other hand, $\pi^{-1}\sigma_i>\pi^{-1}\sigma_k$, so $\pi$ does not invert the values $\sigma_i$ and $\sigma_k$, but this contradicts the fact that $\sigma \leq \pi$.
    \item[(c)] Suppose there exist some points $(\sigma_i,\pi^{-1}\sigma_{i})$ and $(\sigma_{i+1},\pi^{-1}\sigma_{i+1})$ with $\sigma_{i+1}<\sigma_i$ and $\pi^{-1}\sigma_{i+1}<\pi^{-1}\sigma_i$. Proceed with the argument in part (b), but replace $i+1$ with $k$.
    \item[(d)] In the ``only-if" direction: If $\sigma_i < \sigma_{i+1}$ and $\pi^{-1} \sigma_i < \pi^{-1}\sigma_{i+1}$ then both $\sigma$ and $\pi$ do not invert the values $\sigma_i$ and $\sigma_{i+1}$, so the swap $s_i$ is forbidden. Conversely, if the swap $s_i$ is forbidden, then the values $\sigma_i,\sigma_{i+1}$ are inverted neither in $\sigma$ nor in $\pi$, so $\sigma_i<\sigma_{i+1}$ and $\pi^{-1}\sigma_i<\pi^{-1}\sigma_{i+1}$, and the vector $(\sigma_i,\pi^{-1}\sigma_i)\to(\sigma_{i+1},\pi^{-1}\sigma_{i+1})$ points up-right. If a swap $s_i$ is forbidden, then $A_i=B_i=\emptyset$ by Lemma~\ref{lem:forbidden-label}, so the vector can cross neither $x=\sigma_k$ nor $y=\pi^{-1}\sigma_k$ by part (a). 
    \item[(e)] Consider $(\sigma_i,\pi^{-1}\sigma_{i})$ in Quadrant II and $(\sigma_{i+1},\pi^{-1}\sigma_{i+1})$ in Quadrant III. The vector between these two points must point down, and cannot point down-left, so must point down-right. Thus, $\sigma_i<\sigma_{i+1}<\sigma_k$ and $\pi^{-1}\sigma_{i+1}<\pi^{-1}\sigma_k<\pi^{-1}\sigma_i$. So $\sigma\lessdot \sigma s_i$, and $\sigma_k\in B_i\setminus A_i$. Thus, $\sigma_k$ decreases the sum $D$ and $S_{i \sigma_k}=-1$. 

    Likewise, consider $(\sigma_i,\pi^{-1}\sigma_{i})$ in Quadrant III and $(\sigma_{i+1},\pi^{-1}\sigma_{i+1})$ in Quadrant II. The vector between these two points must point up, and cannot point up-right since it crosses quadrants, so must point up-left. Thus, $\sigma_{i+1}<\sigma_i<\sigma_k$ and $\pi^{-1}\sigma_{i}<\pi^{-1}\sigma_k<\pi^{-1}\sigma_{i+1}$. So $\sigma s_i\lessdot \sigma$, and $\sigma_k\in B_i\setminus A_i$. Thus, $\sigma_k$ increases the sum $D$ and $S_{i \sigma_k} = +1$. 
    \item[(f)] This is analogous to part (e).
    
    \item[(g)] Consider the vector for the swap $s_i$ in column $\sigma_k$. If this vector has endpoints in the same quadrant, then $\sigma_k$ is in neither $A_i$ nor $B_i$ so $S_{i\sigma_k}=0.$  If the vector moves between Quadrants II and IV, then the vector crosses both $x=\sigma_k$ and $y=\pi^{-1}\sigma_k,$ so $\sigma_k$ is in both $A_i$ and $B_i$ so $S_{i\sigma_k}=0.$ 
\end{itemize}
\end{proof}

\begin{lem}
\label{lem:correction-term}
\text{}
\begin{itemize}
    \item[(a)] The sum $\sum_i iS_{i\sigma_k}$ for column $\sigma_k$ is the number of points in Quadrant III plus a correction term, $-(k-1)S_{(k-1)0}$.
    \item[(b)] For a permutation path corresponding to column $j=\sigma_k$, the number of points in Quadrant III is the number of elements in $\pi$ that are less than $\sigma_k$ and not inverted with $\sigma_k$.
\end{itemize}
\end{lem}

\begin{proof}

As we move along the permutation path for column $\sigma_k$, the signs of nonzero elements of the sequence $\{S_{i\sigma_k}\}_i$ alternate, with a $+1$ every time we leave Quadrant III, and a $-1$ every time we enter, by parts (e) and (f) of Lemma ~\ref{lem:perm-path-facts}. Pair these with $-1$ followed by $+1$. Each time we leave Quadrant III, the sum of the previous two terms $i(-1) + i'(+1)$ (or just the previous term, if the path started in Quadrant III and this is the first exit) counts the length of that visit to Quadrant III, $i'-i$. 

Suppose $(\sigma_{k-1},\pi^{-1}\sigma_{k-1})$ is not in Quadrant III. Then every point in Quadrant III is followed by an exit, so all are counted by the sum of consecutive $-1,+1$ pairs in nonzero entries of $\{S_{i\sigma_k}\}_i$. So, $\sum_i iS_{i\sigma_k}$ is the number of points in Quadrant III. 

Because $(\sigma_{k-1}, \pi^{-1} \sigma_{k-1})$ is not in Quadrant III, $(\sigma_k,\pi^{-1}\sigma_k)$ is not to the upper right of it, so $s_{k-1}$ is not forbidden by part (d) of Lemma ~\ref{lem:perm-path-facts} and $S_{(k-1)0}=0$. Thus the correction term is $0$. The number of points in Quadrant III is $\sum_i iS_{i\sigma_k} = \sum_i iS_{i\sigma_k} + (k-1)S_{(k-1)0}.$

Suppose $(\sigma_{k-1},\pi^{-1}\sigma_{k-1})$ is in Quadrant III. Then the last entrance to Quadrant III (if any) is an unpaired $-1$. If we append a $+1$ at the end with weight $(k-1)$ to complete the pair, then every point in Quadrant III is counted.

Since $(\sigma_{k-1},\pi^{-1}\sigma_{k-1})$ is in Quadrant III, $(\sigma_k,\pi^{-1}\sigma_k)$ is up and to the right, so $s_{k-1}$ is a forbidden swap and $S_{(k-1)0}=+1$ by part (d) of Lemma ~\ref{lem:perm-path-facts}. Thus the number of points in Quadrant III is $\sum_i i S_{i \sigma_k} + (k-1) = \sum_i iS_{i\sigma_k} + (k-1)S_{(k-1)0},$ proving (a).

Quadrant III contains points $(\sigma_i,\pi^{-1}\sigma_i)$ such that $\sigma_i<\sigma_k,\pi^{-1}\sigma_i<\pi^{-1}\sigma_k,$ and $i<k$. We show this third condition is redundant given the first two.

Suppose that $(\sigma_i,\pi^{-1}\sigma_i)$ such that $\sigma_i<\sigma_k,\pi^{-1}\sigma_i<\pi^{-1}\sigma_k,$ and $i>k$. Then, the values $\sigma_i$ and $\sigma_k$ are inverted in $\sigma$ but not in $\pi$, but $\pi\ge \sigma$, a contradiction. So the points in Quadrant III are all $(\sigma_i,\pi^{-1}\sigma_i)$ such that $\sigma_i<\sigma_k,\pi^{-1}\sigma_i<\pi^{-1}\sigma_k$, proving (b).
\end{proof}

Now we can evaluate $D$ and prove the proposition. Writing $k=\sigma^{-1}j$, we have: 

\begin{align*}
    D = & \sum_{i=1}^{n-1} \sum_{j=0}^{n} i S_{ij}\\
    = & \sum_{j=1}^{n} \left((k-1)S_{(k-1) 0}+ \sum_{i=1}^{n-1} iS_{ij}  \right)\\ 
    = & \sum_{j=1}^n \left|\{i~\vert~i<j,\pi^{-1}i<\pi^{-1}j \}\right|\\
    = & {n \choose 2} - \left|\inv(\pi)\right| \\
    = & {n \choose 2} - \ell(\pi).
\end{align*}
where the first three equalities follow from (\ref{eq:D-from-grid}) and Lemma~\ref{lem:correction-term} (a) and (b), respectively.  Since 
\[
(E F-F E)_{\sigma,\sigma} - (E_{w_0}F-FE_{w_0})_{\sigma,\sigma} = D = {n\choose 2}-\ell(\pi),
\]
we have computed that $(E F-F E)_{\sigma,\sigma} = 2 \ell(\sigma)-\ell(\pi)$ as desired.
\end{proof}

\subsection{Proof of Proposition \ref{prop:off-diag}}
\label{sec:proof-of-off-diag}
Throughout this section, let $\sigma \neq \tau$ be distinct elements of $[e,\pi]_R$.

\begin{proof}[Proof of Proposition \ref{prop:off-diag}]

It suffices to check cases where $(E F)_{\sigma,\tau}\neq 0$ or where $(FE)_{\sigma,\tau}\neq 0$.  We will see that in either case the two are equal, and thus that $(EF-FE)_{\sigma, \tau}=0$.

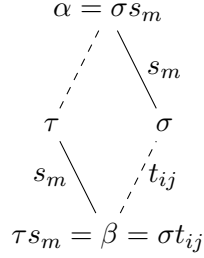
\begin{figure}
    \centering
    \begin{tikzpicture}[scale=1.5]
    \node (tau) at (0,0) {$\tau$};
    \node (sigma) at (1,0) {$\sigma$};
    \node (alpha) at (0.5,1) {$\alpha=\sigma s_m$};
    \node (beta) at (0.5,-1) {$\tau s_m =\beta=\sigma t_{ij}$};
    \draw[dashed] (tau)--(alpha);
    \draw[dashed] (beta)-- node[right] {$t_{ij}$} (sigma);
    \draw (tau)-- node[left] {$s_m$} (beta);
    \draw (alpha)-- node[right] {$s_m$} (sigma);
    \end{tikzpicture}
    \caption{The elements $\sigma, \tau, \alpha, \beta$ used in the proofs of Lemma~\ref{lem:existence-of-up-down} and Proposition~\ref{prop:off-diag}.  The dashed edges indicate cover relations in the strong order and the solid edges cover relations in the weak order.}
    \label{fig:up-down-path}
\end{figure}

\begin{lem}
\label{lem:existence-of-up-down}
There is an element $\alpha \in [e,\pi]_R$ such that $\tau \precdot \alpha \gtrdot \sigma$ if and only if there is an element $\beta \in [e,\pi]_R$ such that $\tau \gtrdot \beta \precdot \sigma$; the elements $\alpha, \beta$ are unique if they exist (see Figure~\ref{fig:up-down-path}).
\end{lem}
\begin{proof}
If $\tau \precdot \alpha \gtrdot \sigma$, the existence and uniqueness of $\beta$ follows from Lemma 2.2 of \cite{sperner-paper}.  So suppose that $\tau \gtrdot \beta=\tau s_m \precdot \beta t_{ij} = \sigma$. Lemma 2.2 of \cite{sperner-paper} again implies that $\alpha= \sigma s_m$ satisfies $\tau \precdot \alpha \gtrdot \sigma$ and is the unique permutation with this property; it remains to verify that $\alpha \in [e,\pi]_R$.

We will show that $\inv(\alpha) \subseteq \inv(\pi).$ Since $\alpha$ covers $\sigma$ in the weak order, $\inv(\alpha)=\inv(\sigma) \cup\{(\alpha_m,\alpha_{m+1}) \},$ and we know $\inv(\sigma) \subseteq \inv(\pi)$ so we only need to check that $(\alpha_m,\alpha_{m+1}) \in \inv(\pi).$ We consider the possible overlaps of $\{i,j\}$ with $\{m,m+1\}$ (the two sets cannot be equal since $\tau \ne \sigma$):
\begin{itemize}
    \item If $\{i,j\}$ is disjoint from $\{m,m+1\}$ then $\tau$ inverts $\alpha_m$ and $\alpha_{m+1}$ so $\pi > \tau$ must as well.
    \item If $i=m$ or $j=m+1$ then $\beta$ inverts $\alpha_m$ and $\alpha_{m+1}$, so $\pi>\beta$ must as well.
    \item If $j=m$ then in positions $i < m<m+1$, $\sigma$ has the values $(\alpha_i, \alpha_{m+1}, \alpha_m).$ The values of $\beta$ and $\tau$ are also determined, $(\alpha_{m+1},\alpha_i, \alpha_{m})$ and $(\alpha_{m+1},\alpha_m, \alpha_i)$, respectively. None of the other permutations invert $(\alpha_m,\alpha_{m+1}).$ However, it must be that $\alpha_{m+1}<\alpha_i <\alpha_m$, and $\tau$ inverts $(\alpha_m,\alpha_i)$ while $\sigma$ inverts $(\alpha_i,\alpha_{m+1}).$ So, $\pi$ must invert both of those pairs, hence by transitivity, $\pi$ inverts $(\alpha_m,\alpha_{m+1}).$
    \item The case $i=m+1$ is similar to the previous case. 
\end{itemize}    
\end{proof}
Using the same definitions of $\alpha = \sigma s_m$ and $\beta = \sigma t_{ij}$ as above, we note that the down weights for edges $\tau \to \beta = \tau s_m$ and $\alpha \to \sigma = \alpha s_m$ are both $m$. To verify that $(E F)_{\sigma,\tau} = (FE )_{\sigma,\tau}$, we need to check that the up weights for edges $\beta \to \sigma$ and $\tau \to \alpha$ agree. We consider the same five cases of possible overlaps of $\{m,m+1\}$ and $\{i,j\}$ as above:

\textbf{Case 1: $\{m,m+1\}$ is disjoint from $\{i,j\}.$ } Since disjoint transpositions commute, $\alpha= \tau s_mt_{ij}s_m=\tau t_{ij}$. 
We have that \begin{align*}\wt^\pi(\tau,\alpha)=1+&|\{k > j ~ \vert ~ \min(\tau_i,\tau_j) < \tau_k < \max(\tau_i,\tau_j)\}|\\+ &|\{k > j ~ \vert ~ \min{(\pi^{-1}\tau_i,\pi^{-1}\tau_j)} < \pi^{-1}\tau_k < \max(\pi^{-1}\tau_i,\pi^{-1}\tau_j) \}| \\=1+&|\{\tau_k~\vert~k>j,~ \tau_i < \tau_k < \tau_j\}| \\ +&|\{ \tau_k ~\vert~ k > j,~ \pi^{-1} \tau_j < \pi^{-1} \tau_k < \pi^{-1} \tau_i\}|.\end{align*}
Also, \begin{align*}\wt^\pi(\beta,\sigma)=1+&|\{k > j ~ \vert ~ \min(\beta_i,\beta_j) < \beta_k < \max(\beta_i,\beta_j)\}|\\+ &|\{k > j ~ \vert ~ \min(\pi^{-1}\beta_i,\pi^{-1}\beta_j) < \pi^{-1}\beta_k < \max(\pi^{-1}\beta_i,\pi^{-1}\beta_j) \}| \\= 1 + &|\{\beta_k ~\vert~ k>j,~ \beta_i < \beta_k < \beta_j\}| \\ + &|\{\beta_k ~\vert~k>j, ~\pi^{-1}\beta_j < \pi^{-1}\beta_k < \pi^{-1}\beta_i\}|.\end{align*}
For all $k \not\in \{m,m+1\},$ $\tau_k =\beta_k$. In particular, $\tau_i = \beta_i$ and $\tau_j = \beta_j$. Since $m$ and $m+1$ are adjacent and unequal to $i,j$, $\{\tau_m,\tau_{m+1}\} = \{\beta_m,\beta_{m+1}\}$ and either both $m,m+1 > j$ or neither is. So, the second expressions for $\wt^\pi(\tau,\alpha)$ and $\wt^\pi(\beta,\sigma)$ count the same sets, so they are equal. 

\textbf{Case 2: $i = m$.} In this case, $t_{ij}=(m \: j)$, $s_m = (m \: m+1)$, and $s_m t_{ij} s_m = (m+1 \: j).$ Thus \begin{align*}\wt^\pi(\tau,\alpha) = 1 + &|\{\tau_k ~\vert~ k>j,~ \tau_{m+1} < \tau_k < \tau_j\}| \\+ &|\{\tau_k ~\vert~ k>j,~ \pi^{-1}\tau_j < \pi^{-1}\tau_k < \pi^{-1}\tau_{m+1} \}|\\
\wt^\pi(\beta,\sigma) = 1 + &|\{\beta_k ~|~ k > j,~ \beta_m < \beta_k < \beta_j\}| \\ + &|\{\beta_k ~\vert~ k>j,~ \pi^{-1}\beta_j < \pi^{-1} \beta_k < \pi^{-1} \beta_m \}|.\end{align*}

The permutations agree except on the set $\{m,m+1,j\}$. We have $\tau_{m+1} = \beta_m = \alpha_j$ and $\tau_j = \beta_j = \alpha_{m+1}$. Since for all $k>j, \beta_k = \tau_k,$ the sets in the expressions for $\wt^\pi(\tau,\alpha)$ and $\wt^\pi(\beta,\sigma)$ are identical.

\textbf{Case 3: $i = m + 1$.} In this case, $t_{ij}=(m+1 \: j)$, $s_m=(m \: m+1)$, and $s_mt_{ij}s_m=(m \: j).$ Thus \begin{align*}\wt^\pi(\tau,\alpha) = 1 + &|\{\tau_k ~\vert~ k>j,~ \tau_m < \tau_k < \tau_j\}| \\+ &|\{\tau_k ~\vert~ k>j,~ \pi^{-1}\tau_j < \pi^{-1}\tau_k < \pi^{-1}\tau_m \}|\\
\wt^\pi(\beta,\sigma) = 1 + &|\{\beta_k ~|~ k > j,~ \beta_{m+1} < \beta_k < \beta_j\}| \\ + &|\{\beta_k ~\vert~ k>j,~ \pi^{-1}\beta_j < \pi^{-1} \beta_k < \pi^{-1} \beta_{m+1} \}|.\end{align*}

Again, the numbers moved by the transpositions $t_{ij}$ acting on $\beta$ and $s_m t_{ij} s_m$ acting on $\tau$ are the same: $\tau_m = \beta_{m+1} = \alpha_j$ and $\tau_j= \beta_j =\alpha_m$. Since for all $k>j$, $\tau_k= \beta_k,$ the sets counted in $\wt^\pi(\tau,\alpha)$ and $\wt^\pi(\beta,\sigma)$ are identical.

\textbf{Case 4: $j = m$.} In this case, $t_{ij}=(i \: m)$, $s_m=(m \: m+1),s_mt_{ij}s_m=(i \: m+1)$.
Thus \begin{align*}\wt^\pi(\tau,\alpha) = 1 + &|\{\tau_k ~\vert~ k>m+1,~ \tau_i < \tau_k < \tau_{m+1}\}| \\+ &|\{\tau_k ~\vert~ k>m+1,~ \pi^{-1}\tau_{m+1} < \pi^{-1}\tau_k < \pi^{-1}\tau_i \}|\\
\wt^\pi(\beta,\sigma) = 1 + &|\{\beta_k ~|~ k > m,~ \beta_i < \beta_k < \beta_m\}| \\ + &|\{\beta_k ~\vert~ k>m,~ \pi^{-1}\beta_m < \pi^{-1} \beta_k < \pi^{-1} \beta_i \}|.\end{align*}

While $\tau_i = \beta_i = \alpha_{m+1}$ and $\tau_{m+1}=\beta_m = \alpha_i,$ the conditions on $k$ are different. In $\wt^\pi(\tau,\alpha),$ we count where $k>m+1$ but in $\wt^\pi(\beta,\sigma),$ we count where $k>m.$ To show that these are equal, we have to show that $\beta_{m+1} \not \in (\beta_i,\beta_m)$ and $\pi^{-1} \beta_{m+1} \not \in (\pi^{-1}\beta_m,\pi^{-1}\beta_i).$ Other than that possibility, the sets are identical.

The permutations $\tau, \alpha, \sigma, \beta$ agree everywhere except for $\{i,m,m+1\}$. In those positions, $\tau$ has the values $(\alpha_{m+1},\alpha_m,\alpha_i)$, $\beta$ has $(\alpha_{m+1}, \alpha_i, \alpha_m)$, and $\sigma$ has the values $(\alpha_i, \alpha_{m+1}, \alpha_m).$ Since $\tau > \beta$ in the weak order, $\alpha_m>\alpha_i=\beta_m$, thus $\beta_{m+1}=\alpha_m$ does not contribute to the first set. Now, $\alpha_m$ is inverted with $\alpha_i$ in $\tau$ so since $\pi>\tau$, $\pi$ inverts $\beta_{m+1}=\alpha_m$ and $\beta_{m}=\alpha_i$, so $\pi^{-1}\beta_{m+1} = \pi^{-1}\alpha_m$. Thus the sets defining $\wt^\pi(\tau,\alpha)$ and $\wt^\pi(\beta,\sigma)$ are identical.

\textbf{Case 5: $j = m + 1$.} In this case, $t_{ij}=(i \: m+1)$, $s_m=(m \: m+1)$,$s_mt_{ij}s_m=(i \: m)$.
Thus \begin{align*}\wt^\pi(\tau,\alpha) = 1 + &|\{\tau_k ~\vert~ k>m,~ \tau_i < \tau_k < \tau_m\}| \\+ &|\{\tau_k ~\vert~ k>m,~ \pi^{-1}\tau_{m} < \pi^{-1}\tau_k < \pi^{-1}\tau_i \}|\\
\wt^\pi(\beta,\sigma) = 1 + &|\{\beta_k ~|~ k > m+1,~ \beta_i < \beta_k < \beta_{m+1}\}| \\ + &|\{\beta_k ~\vert~ k>m+1,~ \pi^{-1}\beta_{m+1} < \pi^{-1} \beta_k < \pi^{-1} \beta_i \}|.\end{align*}

In this case $\tau_i = \beta_i = \alpha_m$ and $\tau_m = \beta_{m+1} = \alpha_i$. Like the previous case, the sets of possible $k$ values differ by one, so we must verify that $\tau_{m+1}$ is not in $(\tau_i,\tau_m)$ nor is $\pi^{-1}\tau_{m+1}$ in $(\pi^{-1}\tau_m,\pi^{-1}\tau_i).$ Other than the possible inclusion of $\tau_{m+1}$ and $\pi^{-1}\tau_{m+1},$ the sets are identical.

The permutations agree except at the locations $i$, $m$, and $m+1$. In $\tau$, the values at those locations are $(\alpha_m, \alpha_i, \alpha_{m+1}).$ In $\beta,$ $(\alpha_m, \alpha_{m+1}, \alpha_i).$ In $\sigma,$ $(\alpha_i, \alpha_{m+1}, \alpha_m).$ Since $\sigma < \alpha$ in the weak order, $\tau_{m+1} = \alpha_{m+1} < \alpha_m$. So, $\tau_{m+1}$ is not in the interval $(\tau_i, \tau_m) = (\alpha_m,\alpha_i)$. To see that $\pi^{-1}\tau_{m+1}$ is to the right of $(\pi^{-1}\tau_m,\pi^{-1}\tau_i)$, it is already inverted with $\alpha_i$ in $\tau$. So, $\pi > \tau$ implies that $\pi^{-1}\tau_m < \pi^{-1}\tau_{m+1}$, and $\tau_{m+1}$ is not included in the sets determining $\wt^\pi(\tau,\alpha)$.
\end{proof}

\section{Strong order Macdonald identities for Schubert polynomials} \label{sec:schubert}

In this section we generalize techniques from \cite{duality-paper, Hamaker} to prove Theorem~\ref{thm:schubert}.

\subsection{Macdonald's reduced word identity}
Proposition~\ref{prop:applying-nabla} below indicates a connection between Schubert polynomials and the weak order weights appearing in Section~\ref{sec:sl2}.  Let $\nabla=\sum_{i=1}^n \partial/\partial x_i$; the following proposition was the key observation of \cite{Hamaker}:

\begin{prop}[Hamaker, Pechenik, Speyer, and Weigandt \cite{Hamaker}]
\label{prop:applying-nabla}
Let $\sigma \in S_n$, then
\[
\nabla \S_\sigma = \sum_{\sigma s_i \lessdot \sigma} i \: \S_{\sigma s_i}.
\]
\end{prop}

Hamaker, Pechenik, Speyer, and Weigandt \cite{Hamaker} used this result to give a simple new proof of Macdonald's celebrated reduced word formula for the principal specialization of Schubert polynomials; we adapt this proof technique in the proof of Theorem~\ref{thm:schubert} in Section~\ref{sec:padded}.

\begin{theorem}[Macdonald \cite{Macdonald}]
\label{thm:macdonald-identity}
Let $\sigma \in S_n$. Then
\[
\S_\sigma(1,\ldots,1)=\frac{1}{\ell(\sigma)!}\sum_{s_{i_1} \cdots s_{i_{\ell(\sigma)}}=\sigma} i_1 \cdots i_{\ell(\sigma)}.
\]
\end{theorem}

\subsection{$\pi$-padded Schubert polynomials and strong order Macdonald identities}
\label{sec:padded}

Given a $132$-avoiding permutation $\pi \in S_n$, Proposition~\ref{prop:schubert-is-code} says that $\S_\sigma=x^{\beta}$ for some composition $\beta=(\beta_1,\ldots, \beta_n)$ with $\sum_i \beta_i = \ell(\sigma)$ (in fact, this composition is given by the \emph{inversion table} of $\pi$, but we will not need this fact).  We introduce a new set of variables $y_1,\ldots,y_n$ and define the \emph{$\pi$-padded Schubert polynomials} $\S_\sigma^{\pi}$ for $\sigma \leq \pi$ as the images of the usual Schubert polynomials $\S_\sigma$ under the map sending $x^{\alpha} \mapsto x^{\alpha}y^{\beta-\alpha}$ for each monomial $x^{\alpha}$.  That is, we add homogenizing variables $y_i$ to $\S_\sigma$ so that it has degree $\beta_i$ in the variables $x_i,y_i$ for each $i=1,\ldots,n$. Note that, by Proposition~\ref{prop:applying-nabla}, any monomial $x^{\alpha}$ appearing in $\S_\sigma$ has $\alpha \leq \beta$, so $\S_\sigma^{\pi}$ is indeed a polynomial.  In the case $\pi=w_0$, this construction recovers the padded Schubert polynomials of \cite{duality-paper}.  We write
\[
U_n^{\pi}=\spn_{\C}\{\S_\sigma^{\pi} \: | \: \sigma \leq \pi\}.
\]

As noticed in \cite{duality-paper}, it now makes sense to define another differential operator
\[
\Delta = \sum_{i=1}^n x_i \frac{\partial}{\partial y_i},
\]
and to adapt $\nabla$ so that it acts on $U_n^{\pi}$:
\[
\nabla = \sum_{i=1}^n y_i \frac{\partial}{\partial x_i}.
\]
It is immediate that the statement of Proposition~\ref{prop:applying-nabla} holds with this modified $\nabla$ and $\pi$-padded Schubert polynomials.  Interestingly, $\Delta$ applied to a $\pi$-padded Schubert polynomial $\S_\sigma^{\pi}$ may be expanded as a sum over the \emph{strong order} covers of $\sigma$, rather than the weak order appearing in Proposition~\ref{prop:applying-nabla}.

\begin{prop}
\label{prop:applying-delta-to-padded}
For $\sigma, \pi \in S_n$ with $\sigma \leq \pi$ and $\pi$ $132$-avoiding, we have
\[
\Delta \S_\sigma^{\pi} = \sum_{\substack{\sigma t_{ij} \in [e,\pi]_R \\ \sigma \precdot \sigma t_{ij}}} \wt^{\pi}(\sigma,\sigma t_{ij}) \S^{\pi}_{\sigma t_{ij}},
\]
where $\wt^{\pi}(\sigma,\sigma t_{ij})$ is the weight function appearing in Section~\ref{sec:sl2}.
\end{prop}
\begin{proof}
Let $\pi\in S_n$ be $132$-avoiding, with $\S_{\pi}=x^{\beta}$.  Let $V$ be the vector space spanned by all monomials $x^{\alpha}$ with $\alpha \leq \beta$.  Define a linear operator $\tilde{H}$ on $V$ that multiplies each monomial $x^{\alpha}y^{\beta-\alpha}$ by the scalar $\ell(\pi)-2|\alpha|$.  It is elementary to verify that the operators $\Delta, \nabla, \tilde{H}$ acting on $V$ satisfy the relations defining $\mf{sl}_2$.  By Proposition~\ref{prop:applying-nabla}, the subspace $U_n^{\pi} \subseteq V$ is closed under the actions of $\nabla$ and $\tilde{H}$.  Furthermore, since the number of elements in $[e,\pi]_R$ of length $r$ is equal to the number of length $\ell(\pi)-r$ for all $r$ (see \cite{Wei}), $U_n^{\pi}$ is weight symmetric.  Thus Proposition~\ref{prop:b-submodule} implies that $U_n^{\pi}$ is in fact an $\mf{sl}_2$-subrepresentation of $V$.

Let $E,F,H \in \End(\C [e,\pi]_R)$ be the transformations by which $e,f,h$ act in the $\mf{sl}_2$-representation given in Section~\ref{sec:sl2} and let $\ph: \C [e,\pi]_R \to U_n^{\pi}$ denote the invertible linear transformation sending $\sigma \mapsto \S_\sigma^{\pi}$ for all $\sigma \in [e,\pi]_R$.  Then Proposition~\ref{prop:applying-nabla} implies that $\ph$ identifies $F$ with $\nabla$ and the fact that $\S_\sigma(x)$ is homogeneous of degree $\ell(\sigma)$ implies that $\ph$ identifies $H$ with $\tilde{H}$.  Now, Theorem~\ref{thm:at-most-one-action} implies that $E$ must be identified with $\Delta$; by the definition of $E$ this gives the desired formula.
\end{proof}

We are now ready to complete the proof of Theorem~\ref{thm:schubert}, which can be seen as a strong order analogue of Macdonald's weak order formula (Theorem~\ref{thm:macdonald-identity}).  This theorem generalizes the main result of \cite{duality-paper}, which is the case $\pi=w_0$.

\begin{proof}[Proof of Theorem~\ref{thm:schubert}]
Let $\pi \in S_n$ be a $132$-avoiding permutation, and let $\sigma \leq \pi$.  Then it is easy to see that 
\[
\Delta^{\ell(\pi)-\ell(\sigma)} \S_\sigma^{\pi} = (\ell(\pi)-\ell(\sigma))! \: \S_\sigma^{\pi}(1,\ldots,1; 1,\ldots,1) x^{\beta}.
\]
Indeed this fact holds for any polynomial from $U_n^{\pi}$ which is homogeneous in the $x$-variables of degree $\ell(\sigma)$.  On the other hand, repeated application of Proposition~\ref{prop:applying-delta-to-padded} yields
\[
\Delta^{\ell(\pi)-\ell(\sigma)} \S_\sigma^{\pi} = \left( \sum_{C: \sigma \to \pi} \wt^{\pi}(C)\right) x^{\beta}, 
\]
the sum being over all saturated chains from $\sigma$ to $\pi$ in $([e,\pi]_R, \preceq)$.  Since
\[
\S_\sigma^{\pi}(1,\ldots,1; 1,\ldots,1)=\S_\sigma(1,\ldots,1),
\]
this completes the proof.
\end{proof}

\section*{Acknowledgements}

We wish to thank the MIT PRIMES program, during which this research was conducted.

\bibliographystyle{plain}
\bibliography{main}
\end{document}